\documentclass[12pt]{amsart}

\usepackage{amssymb}
\usepackage{bm}
\usepackage{graphicx}
\usepackage{psfrag}
\usepackage[usenames,dvipsnames]{color}
\usepackage{url}
\usepackage{algpseudocode}
\usepackage{enumerate}
\usepackage{fancyhdr}
\usepackage{xy}
\input xy
\xyoption{all}

\newtheorem{theorem}{Theorem}[section]
\newtheorem{proposition}[theorem]{Proposition}
\newtheorem{lemma}[theorem]{Lemma}
\newtheorem{corollary}[theorem]{Corollary}

\theoremstyle{definition}
\newtheorem{definition}[theorem]{Definition}
\newtheorem{notation}[theorem]{Notation}
\newtheorem{example}[theorem]{Example}
\newtheorem{remark}[theorem]{Remark}
\numberwithin{equation}{section}

\newcommand\nn{\mathbb{N}}
\newcommand\qq{\mathbb{Q}}
\newcommand\rr{\mathbb{R}}
\newcommand\zz{\mathbb{Z}}

\newcommand\pval{\mathsf{v}_p}

\newcommand{\excise}[1]{}

\begin{document}

	\mbox{}
	\title{The Elasticity of Puiseux Monoids}
	\author{Felix Gotti}
	\address{Mathematics Department\\UC Berkeley\\Berkeley, CA 94720}
	\email{felixgotti@berkeley.edu}
	\author{Christopher O'Neill}
	\address{Mathematics Department\\UC Davis\\Davis, CA 95616}
	\email{coneill@math.ucdavis.edu}
	\date{\today}
	
	\begin{abstract}
		Let $M$ be an atomic monoid and let $x$ be a non-unit element of $M$. The elasticity of $x$, denoted by $\rho(x)$, is the ratio of its largest factorization length to its shortest factorization length, and it measures how far is $x$ from having a unique factorization. The elasticity $\rho(M)$ of $M$ is the supremum of the elasticities of all non-unit elements of $M$. The monoid $M$ has accepted elasticity if $\rho(M) = \rho(m)$ for some $m \in M$. In this paper, we study the elasticity of Puiseux monoids (i.e., additive submonoids of $\qq_{\ge 0}$). First, we characterize the Puiseux monoids $M$ having finite elasticity and find a formula for $\rho(M)$. Then we classify the Puiseux monoids having accepted elasticity in terms of their sets of atoms. When $M$ is a primary Puiseux monoid, we describe the topology of the set of elasticities of $M$, including a characterization of when $M$ is a bounded factorization monoid. Lastly, we give an example of a Puiseux monoid that is bifurcus (that is, every nonzero element has a factorization of length at most $2$).
	\end{abstract}
	
	\maketitle
	
	\section{Introduction}
	\label{sec:intro}
	
	Non-unique factorizations of monoid elements $m \in M$ is often studied using arithmetic invariants that measure how far from unique the factorizations of $m$ are. One of the most popular factorization invariants is elasticity, computed as the ratio of the largest factorization length to the shortest, which provides a concise (albeit coarse) measure of the non-unique factorizations of $m$. The elasticity of $M$, computed as the supremum of the elasticities obtained within $M$, provides a global measure of the non-uniqueness of factorizations within $M$.
	
	Puiseux monoids, which are additive submonoids of $\qq_{\ge 0}$, are a recent addition to the factorization theory literature. Introduced as a generalization of numerical monoids (additive submonoids of the natural numbers), Puiseux monoids are a natural setting in which to study non-unique factorization theory. Indeed, questions in factorization theory often involve realization of various pathological behaviors (see, for instance,~\cite{CGP14} and~\cite{GS17}), and Puiseux monoids offer a wide variety of such behaviors. Many known results for numerical monoids, including concise characterizations of some invariants \cite{CHM06} and pathological behaviors of others \cite{ACHP07}, are due largely in part to the Archimedean property, i.e., that any two nonzero elements have a common multiple. Puiseux monoids are also Archimedean, but need not be finitely generated like numerical monoids, resulting in a wider array of possible behaviors.
	
	To date, the study of Puiseux monoids has centered around classifying their atomic structure~\cite{fG17,GG16}. In this paper, we initiate the study of non-unique factorization invariants for the family of Puiseux monoids, focusing specifically on the elasticity invariant. After giving some preliminary definitions in Section~\ref{sec:background}, we characterize which Puiseux monoids have finite elasticity (Theorem~\ref{t:pmelast}), a result which closely resembles the known formula for numerical monoids~\cite{CHM06}. We also specify which Puiseux monoids have accepted elasticity (Theorem~\ref{t:pmaccept}). In Sections~\ref{sec:ppmbf} and~\ref{sec:ppmelasticity}, we focus our attention on primary Puiseux monoids (Definition~\ref{d:primary}), characterizing when a primary Puiseux monoid $M$ is a BF-monoid (Theorem~\ref{t:ppmbf}) and describing the set of elasticities of $M$ (Theorems~\ref{t:finiteunstable} and~\ref{t:infiniteunstable}). We conclude with an example of a bifurcus Puiseux monoid in Section~\ref{sec:bifurcus}.

	\section{Background}
	\label{sec:background}
	
	In this section, we fix notation and establish the nomenclature we will use later. To do this, we recall some basic definitions related to commutative monoids and their sets of atoms. Reference material on commutative monoids can be found in Grillet~\cite{pG01}. In addition, the monograph of Geroldinger and Halter-Koch~\cite{GH06} offers extensive background information on atomic monoids and non-unique factorization theory.
	
	\begin{notation}\label{n:numdenom}
		The double-struck symbols $\mathbb{N}$ and $\mathbb{N}_0$ denote the sets of positive integers and non-negative integers, respectively. If $r \in \qq_{> 0}$, then the unique $a,b \in \nn$ such that $r = a/b$ and $\gcd(a,b)=1$ are denoted by $\mathsf{n}(r)$ and $\mathsf{d}(r)$, respectively.
	\end{notation}

	The unadorned term \emph{monoid} always refers to a commutative cancellative monoid, say $M$, which is always written additively; in particular, ``$+$'' denotes the operation of $M$, while $0$ denotes the identity element. We write $M^\bullet$ for the set of nonzero elements of $M$. For $a,b \in M$, we say that $a$ \emph{divides} $b$ in $M$ if there exists $a' \in M$ such that $a + a' = b$. An element dividing $0$ in $M$ is called a \emph{unit}, and $M$ is called \emph{reduced} if its only unit is $0$. Every monoid in this paper is reduced.
	
	\begin{definition}\label{d:factorizations}
		An element $a \in M^\bullet$ is \emph{irreducible} or an \emph{atom} if whenever $a = b + c$, either $b$ or $c$ is a unit. We denote the set of atoms of $M$ by $\mathcal{A}(M)$, and call $M$ \emph{atomic} if $M = \langle \mathcal{A}(M) \rangle$. A \emph{factorization} of an element $x \in M^\bullet$ is a formal sum
		\[
			a_1 + \dots + a_m
		\]
		of atoms adding to $x$. We write $\mathsf{Z}(x)$ for the \emph{set of factorizations} of~$x$. If~$|\mathsf Z(x)| < \infty$ for all $x \in M^\bullet$, we say $M$ is an \emph{FF-monoid} (\emph{finite factorization monoid}).  
	\end{definition}
		
	\begin{definition}\label{d:elasticity}
		Given a monoid element $x \in M$, the \emph{length} of a factorization $z \in \mathsf Z(x)$, denoted $|z|$, is the number of formal summands in $z$. In addition,
		\[
			\mathsf{L}(x) = \{|z| : z \in \mathsf{Z}(x)\}
		\]
		is called \emph{length set} of $x$.  If $M$ is atomic and $|\mathsf{L}(x)| < \infty$ for every $x \in M$, we say that $M$ is a \emph{bounded factorization} monoid or, simply, a \emph{BF-monoid}. The quotient
		\[
			\rho(x) = \frac{\sup \mathsf{L}(x)}{\inf \mathsf{L}(x)}
		\]
		is the \emph{elasticity} of $x$. The \emph{set of elasticities} of $M$ is $R(M) = \{\rho(x) \mid x \in M^\bullet\} \subseteq [1, \infty]$, while the \emph{elasticity} of $M$ is
		\[
			\rho(M) = \sup R(M).
		\]
		The elasticity of $M$ is \emph{accepted} if there exists $x \in M$ such that $\rho(M) = \rho(x)$. Moreover, $M$ is said to be \emph{fully elastic} if $R(M)$ is an interval of rationals.
	\end{definition}
	
	\begin{definition}\label{d:puiseux}
		A \emph{Puiseux monoid} is an additive submonoid $M$ of $\qq_{\ge 0}$. We say $M$ is a \emph{numerical monoid} if it is finitely generated.
	\end{definition}
		
	\begin{remark}\label{r:numerical}
		The definition of numerical monoid given above differs slightly from the standard definition, which requires $M$ to be a cofinite subset of $\nn_0$. However, there is no risk of ambiguity here as every finitely generated Puiseux monoid is isomorphic to a (standardly defined) numerical monoid.
	\end{remark}
	
	\begin{remark}\label{r:puiseux}
		Albeit a natural generalization of numerical monoids, Puiseux monoids need not be atomic; see, for instance, \cite[Example~3.5]{fG17}. For more on the atomic structure of Puiseux monoids, see \cite{fG17} and \cite{GG16}.
	\end{remark}
	
	\section{Elasticity}
	\label{sec:elasticity}
	
	In this section, we examine the elasticity $\rho(M)$ of an arbitrary Puiseux monoid $M$. More specifically, we~give a formula for the elasticity $\rho(M)$ in terms of the atoms of $M$ (Theorem~\ref{t:pmelast}), and determine when this elasticity is accepted (Theorem~\ref{t:pmaccept}). The next proposition is used in the proof of Theorem~\ref{t:pmelast}.
	
	\begin{proposition} \label{p:zerolimitbf}
		Let $M$ be a Puiseux monoid. If $0$ is not a limit point of $M$, then~$M$ is a BF-monoid.
	\end{proposition}
	
	\begin{proof}
		The atomicity of $M$ when $0$ is not a limit point is precisely \cite[Theorem~3.10]{fG17}. So it suffices to verify that every element of $M$ has bounded length set. Suppose, by way of contradiction, that $M$ has an element $x$ with $|\mathsf{L}(x)| = \infty$. For any $\epsilon > 0$, there exists $n > x/\epsilon$ such that $x = a_1 + \dots + a_n$ for some $a_1, \dots, a_n \in \mathcal{A}(M)$. Assuming $a_1 \le a_i$ for $i=1,\dots,n$, we obtain 
		\[
			na_1 \le a_1 + \dots + a_n = x,
		\]
		and, therefore, it follows that $a_1 \le x/n < \epsilon$. As such, $0$ is a limit point of $M$, a contradiction.
	\end{proof}
	
	\begin{theorem} \label{t:pmelast}
		Let $M$ be an atomic Puiseux monoid. If $0$ is a limit point of $M$, then $\rho(M) = \infty$. Otherwise,
		\begin{equation} \label{eq:pmelast}
			\rho(M) = \frac{\sup \mathcal{A}(M)}{\inf \mathcal{A}(M)}.
		\end{equation}
	\end{theorem}
	
	\begin{proof}
		First, suppose $\{a_n\}$ is a sequence of atoms decreasing to $0$ and, for every $n \ge 1$, set $x_n = \mathsf{n}(a_n)\mathsf{n}(a_1)$. We can factor each $x_n$ as
		\[
			x_n = \mathsf{n}(a_n)\mathsf{d}(a_1) a_1 = \mathsf{n}(a_1)\mathsf{d}(a_n) a_n,
		\]
		so $\inf \mathsf{L}(x_n) \le \mathsf{n}(a_n)\mathsf{d}(a_1)$ and $\sup \mathsf{L}(x_n) \ge \mathsf{n}(a_1)\mathsf{d}(a_n)$. As a result,
		\[
			\rho(x_n) = \frac{\sup \mathsf{L}(x_n)}{\inf \mathsf{L}(x_n)} \ge \frac{\mathsf{n}(a_1)\mathsf{d}(a_n)}{\mathsf{n}(a_n)\mathsf{d}(a_1)} = \frac{a_1}{a_n},
		\]
		which implies that the sequence $\{\rho(x_n)\}$ is unbounded. Thus, $\rho(M) = \infty$.
		
		Now, suppose $0$ is not a limit point of $M$. Set $\mathsf{A} = \sup \mathcal{A}(M)$ and $\mathsf{a} = \inf \mathcal{A}(M)$.  Let $\{b_n\}$ and $\{c_n\}$ be sequences of atoms converging to $\mathsf{a}$ and $\mathsf{A}$, respectively. Setting $x_n = \mathsf{n}(b_n)\mathsf{n}(c_n)$, we find that
		\[
			\rho(M) \ge \rho(x_n) = \frac{\sup \mathsf{L}(x_n)}{\inf \mathsf{L}(x_n)} \ge \frac{\mathsf{n}(c_n)\mathsf{d}(b_n)}{\mathsf{n}(b_n)\mathsf{d}(c_n)} = \frac{c_n}{b_n}.
		\]
		Taking the limit when $n \to \infty$ yields $\rho(M) \ge \mathsf{A}/\mathsf{a}$. Therefore it only remains to show that $\rho(M) \le \mathsf{A}/\mathsf{a}$. Fix $x \in M^\bullet$ and set $\ell = \min \mathsf{L}(x)$ and $L = \max \mathsf{L}(x)$ (finite by Proposition~\ref{p:zerolimitbf}). If $a_1 + \dots + a_L$ and $a'_1 + \dots + a'_\ell$ are factorizations in $\mathsf{Z}(x)$, it follows that
		\[
			L \mathsf{a} \le L \min_{1 \le j \le L} \{a_j\} \le a_1 + \dots + a_L = a'_1 + \dots + a'_\ell  \le \ell \max_{1 \le i \le \ell} \{a'_i\} \le \ell \mathsf{A}.
		\]
		This implies that $\rho(x) = L/\ell \le \mathsf{A}/\mathsf{a}$, so we obtain $\rho(M) \le \mathsf{A}/\mathsf{a}$, as desired.
	\end{proof}
	
	\begin{remark} \label{r:numelast}
		If $\mathcal{A}(M)$ is finite, then $M$ is a numerical monoid; in this case, Theorem~\ref{t:pmelast} coincides with the elasticity formula given in \cite[Theorem~2.1]{CHM06}.
	\end{remark}

	Recall that the elasticity of $M$ is accepted if there exists $x \in M$ with $\rho(x) = \rho(M)$.
	
	\begin{theorem} \label{t:pmaccept}
		For any atomic Puiseux monoid $M$ such that $\rho(M) < \infty$, the elasticity of $M$ is accepted if and only if $\mathcal{A}(M)$ has both a maximum and a minimum.
	\end{theorem}
	
	\begin{proof}
		Let $a = \sup \mathcal{A}(M)$ and $b = \inf \mathcal{A}(M)$. By Theorem \ref{t:pmelast}, we have $\rho(M) = a/b$. For the direct implication, suppose $x \in M^\bullet$ satisfies $\rho(x) = \rho(M)$. Set $L = \max \mathsf L(x)$ and $\ell = \min \mathsf L(x)$, and let $a_1 + \dots + a_L$ and $a'_1 + \dots + a'_\ell$ be two factorizations in $\mathsf{Z}(x)$.  Since $L/\ell = a/b$ by assumption, every inequality in the expression
		\[
			Lb \le L \min_{1 \le j \le L} \{a_j\} \le a_1 + \dots + a_L = a'_1 + \dots + a'_\ell \le \ell \max_{1 \le i \le \ell} \{a'_i\} \le \ell a
		\]
		is, indeed, an equality. Hence each $a_i = a$ and each $a_j' = b$. As such, $a, b \in \mathcal A(M)$. Conversely, suppose that $a, b \in \mathcal{A}(M)$. Setting $x = \mathsf{n}(a)\mathsf{n}(b)$ yields
		\[
			\rho(x) = \frac{\sup \mathsf{L}(x)}{\inf \mathsf{L}(x)} \ge \frac{\mathsf{n}(a)\mathsf{d}(b)}{\mathsf{n}(b)\mathsf{d}(a)} = \frac{a}{b} = \rho(M),
		\]
		so $\rho(M) = \rho(x)$.
	\end{proof}
	
	Corollary~\ref{c:pmaccept} follows immediately from the proof of Theorem~\ref{t:pmaccept} and matches the analogous result for numerical semigroups \cite{CHM06}.
	
	\begin{corollary} \label{c:pmaccept}
		If the elasticity of $M$ is both finite and accepted, then $x \in M$ has elasticity $\rho(M)$ if and only if $x$ is an integer multiple of both $\min \mathcal A(M)$ and $\max \mathcal A(M)$.
	\end{corollary}
	
	\section{Primary Puiseux monoids}
	\label{sec:ppmbf}
	
	In this section and the next, we focus our attention on primary Puiseux monoids. Let us recall that $M$ is a BF-monoid if it is atomic and $\mathsf{L}(x)$ is finite for all $x \in M$. Not every atomic Puiseux monoid is a BF-monoid; see, for instance, Example~\ref{e:factorial1op}. The main result in this section is Theorem~\ref{t:ppmbf}, which characterizes those primary Puiseux monoids that are BF-monoids.
	
	\begin{definition}\label{d:primary}
		A \emph{primary} Puiseux monoid is a monoid of the form 
		\[
			\big\langle a_p/p \ \big{|} \ p \in P \ \text{ and } \ a_p \in \nn \setminus p\nn \big\rangle \subset \qq_{\ge 0}
		\]
		for some set of primes $P$.  Notice that if $a_p = 1$ for each $p \in P$, we recover the Puiseux monoids introduced in \cite[Definition 5.1]{GG16}, formerly called primary Puiseux monoids.
	\end{definition}
	
	\begin{example}\label{e:bfplot}
		Let $p_1, p_2, \dots$ be an increasing enumeration of the prime numbers, and consider the Puiseux monoid
		\[
			M = \bigg\langle \frac 12, \frac{p_n + 1}{p_n} \ \bigg{|} \ n \ge 2 \bigg\rangle.
		\]
		Figure~\ref{f:primarybf} displays the elasticities of $m \in M$ for $\rho(m) > 1$. Since $\min\mathcal A(M) = \frac{1}{2}$ and $\max\mathcal A(M) = \frac{4}{3}$, the elasticity $\rho(M) = \frac{8}{3}$ is accepted by Theorem~\ref{t:pmaccept}. Moreover, since every atom of $M$ is unstable (Definition~\ref{d:stable}), we will see in Theorem~\ref{t:infiniteunstable} that the set $R(M)$ is dense in $[1, \frac{8}{3}]$.
	\end{example}
	
	\begin{figure}[h]
	\begin{center}
		\includegraphics[width=6.0in]{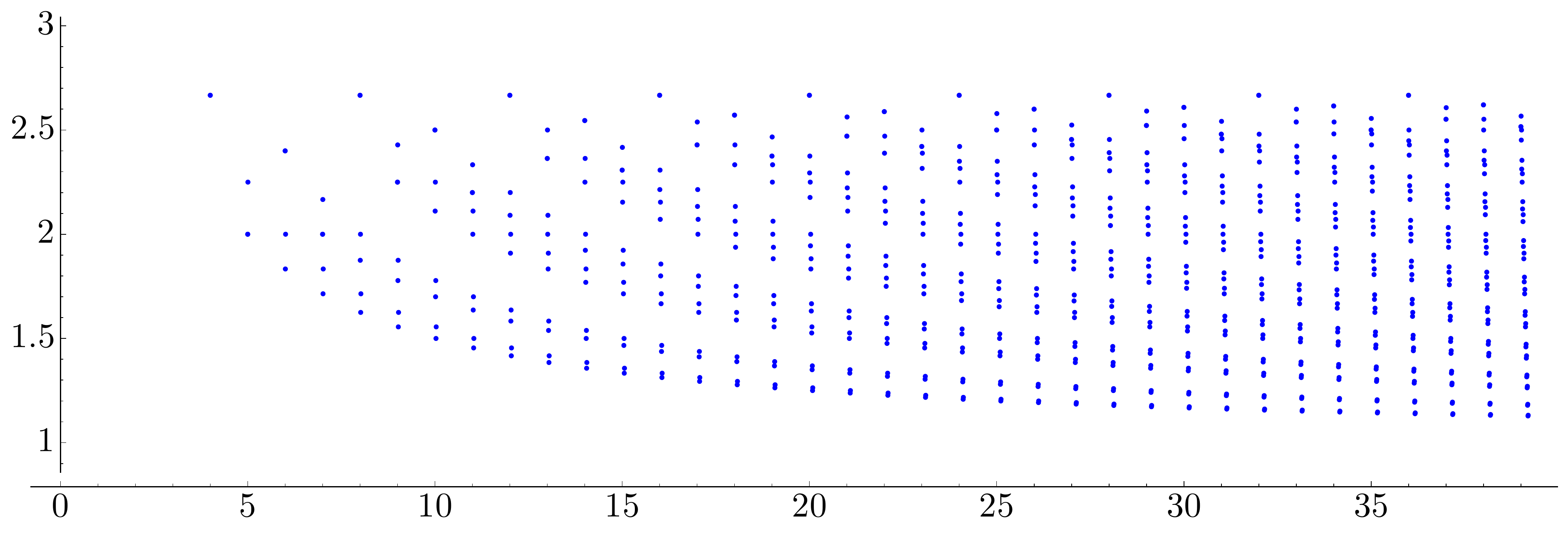}
	\end{center}
	\caption{A plot depicting the elasticities of those elements $m \in M$ with $\rho(m) > 1$ from Example~\ref{e:bfplot}.}
	\label{f:primarybf}
	\end{figure}

	\vspace{10pt}
	
	For a prime $p$ and a natural $n$, let $\pval(n)$ denote the exponent of the maximal power of $p$ dividing $n$ (define $\pval(0) = \infty$). The $p$-\emph{adic valuation map} on $\qq$ is the map defined by $\pval(0) = \infty$ and $\pval(r) = \pval(\mathsf{n}(r)) - \pval(\mathsf{d}(r))$ for $r \neq 0$. The next elementary result, whose proof follows immediately by taking each $p_i$-valuation to both sides of \eqref{eq:elementary lemma}, is frequently used throughout the paper.
	
	\begin{lemma}\label{l:ppm}
		Fix $x, a_1, \ldots a_n \in \zz$, and distinct primes $p_1, \ldots, p_n$. If
		\begin{equation} \label{eq:elementary lemma}
			x = \sum_{i=1}^n \frac{a_i}{p_i},
		\end{equation}
		then $p_i \mid a_i$ for each $i$.
	\end{lemma}
	
	
	\begin{definition} \label{d:stable}
		Let $M$ be a Puiseux monoid. We say $a \in \mathcal{A}(M)$ is \emph{stable} if the set 
		\[
			\{x \in \mathcal{A}(M) \mid \mathsf{n}(x) = \mathsf{n}(a)\}
		\]
		is infinite, and \emph{unstable} otherwise. The submonoid of $M$ generated by the stable (resp., unstable) atoms is denoted by $\mathcal S(M)$ (resp., $\mathcal U(M)$) and is called the \emph{stable} (resp., \emph{unstable}) submonoid of $M$. The elements of $\mathcal S(M)$ and $\mathcal U(M)$ are called \emph{stable} and \emph{unstable}, respectively. We call an element \emph{absolutely unstable} if it has no nonzero stable divisors in $M$. Notice that an element might be stable and unstable simultaneously.
	\end{definition}
	
	\begin{remark}\label{r:factorial1op}
		Example~\ref{e:factorial1op} demonstrates how stable atoms affect factorization lengths. Although Example~\ref{e:factorial1op} follows from results in Section~\ref{sec:ppmelasticity}, we include a thorough discussion here because the techniques used therein appear throughout the next two sections.
	\end{remark}
	
	\begin{example} \label{e:factorial1op}
		Let $P$ be an infinite set of primes, and $M = \langle 1/p \mid p \in P \rangle$. For $x \in M^\bullet$\!, we claim $|\mathsf{L}(x)| < \infty$ if and only if $|\mathsf Z(x)| = 1$. The key observation is that $\mathsf{L}(1) = P$. Clearly $P \subseteq \mathsf{L}(1)$ since $1 = p(1/p)$ for each $p \in P$. On the other hand, any factorization
		\[
			a_1 \frac 1{p_1} + \dots + a_n \frac 1{p_n} \in \mathsf{Z}(1)
		\]
		with $a_n > 0$ satisfies $p_n \mid a_n$ by Lemma~\ref{l:ppm}, meaning $a_1 + \dots + a_n = a_n = p_n \in P$.  
		
		More generally, suppose $x \in M$ has two distinct factorizations, say
		\[
			x = a_1 \frac 1{p_1} + \dots + a_m \frac 1{p_m} = b_1 \frac 1{p_1} + \dots + b_m \frac 1{p_m}.
		\]
		By reindexing appropriately, we can assume $a_m < b_m$. Subtracting the two equations and applying Lemma~\ref{l:ppm}, we find that $p_m$ must divide $b_m - a_m$. Therefore $b_m > p_m$,~so
		\[
			x = 1 + (b_m - p_m) \frac 1{p_m} + b_1 \frac 1{p_1} + \dots + b_{m-1} \frac 1{p_{m-1}}.
		\]
		Since $|\mathsf{L}(1)| = \infty$, it follows that $|\mathsf{L}(x)| = \infty$. \\ \\
	\end{example}

	\begin{lemma} \label{l:stable}
		Let $M$ be a Puiseux monoid.
		\begin{enumerate}
		\item 
		If $s \in \mathcal A(M)$ is stable, then $|\mathsf{L}(\mathsf{n}(s))| = \infty$.
		
		\item 
		If $s \in M$ is a stable element with more than one factorization, then $|\mathsf L(s)| = \infty$.
		\end{enumerate}
	\end{lemma}
	
	\begin{proof}
		First, fix a stable atom $s \in \mathcal A(M)$. There exists an infinite sequence of distinct integers $d_1, d_2, \ldots$ such that $\mathsf n(s)/d_i \in \mathcal A(M)$ for each $i \ge 1$. Writing 
		\[
			\mathsf n(s) = d_1 \frac k{d_1} = d_2 \frac k{d_2} = \cdots,
		\]
		we see that $\{d_i \mid i \ge 1\} \subset \mathsf L(\mathsf n(s))$, so $|\mathsf L(\mathsf n(s))| = \infty$.  This proves part~(1).
		
		Next, fix a stable element $s \in M$, and write $M = \langle a_n/p_n \mid n \in \nn \rangle$, where $\{a_n\}$ is a sequence of naturals and $\{p_n\}$ is an increasing sequence of primes such that $p_n \nmid a_n$ for every natural $n$. Suppose $s$ has two distinct factorizations
		\begin{equation} \label{eq:two factorizations of s}
			z = \sum_{j=1}^n \alpha_j \frac{a_j}{p_j} \ \text{ and } \ z' = \sum_{j=1}^n \beta_j \frac{a_j}{p_j},
		\end{equation}
		and assume that $a_j/p_j$ is stable whenever $\alpha_j > 0$. By omitting atoms appearing in both $z$ and $z'$ and replacing $s$ with an appropriate divisor, it suffices to assume $\alpha_j$ and $\beta_j$ are never simultaneously nonzero. Equating each $p_j$-adic valuation of both sums in \eqref{eq:two factorizations of s}, we conclude that $s \in \nn$, at which point Lemma~\ref{l:ppm} and part~(1) imply $|\mathsf L(s)| = \infty$.
	\end{proof}
	
	Theorem~\ref{t:ppmbf} proves that the phenomenon in Lemma~\ref{l:stable} (i.e.,\ stable atoms) is the only obstruction to finite length sets in a primary Puiseux monoid.
	
	\begin{theorem} \label{t:ppmbf}
		For a primary Puiseux monoid $M$, the following are equivalent:
		\begin{enumerate}
			\item $M$ is a FF-monoid;
			\item $M$ is a BF-monoid;
			\item Every $a \in \mathcal A(M)$ is unstable.
		\end{enumerate}
	\end{theorem}
	
	\begin{proof}
		Clearly (1) implies (2), and (2) implies (3) follows from Lemma~\ref{l:stable}. Suppose every atom of $M$ is unstable. It suffices to show that each integer in $M$ has only finitely many factorizations. Fix $x \in \nn \cap M$ and a factorization
		\[
			\sum_{i=1}^m \alpha_i \frac{a_i}{p_i} \in \mathsf{Z}(x),
		\]
		with $a_i/p_i \in \mathcal A(M)$ and $p_i \nmid a_i$. By Lemma~\ref{l:ppm}, $p_i$ must divide $\alpha_i$ for $1 \le i \le m$, and thus each $a_i \le x$. Since $M$ has no stable atoms, only finitely many atoms of $M$ appear in factorizations of $x$, so $\mathsf{Z}(x)$ is finite.
	\end{proof}
	
	We conclude this section with several examples demonstrating that Theorem~\ref{t:ppmbf} does not easily generalize to more general classes of Puiseux monoids.
	
	\begin{example}\label{e:bfnotff}
		Let $p_1, p_2, \dots$ be an enumeration of the odd prime numbers, and consider the Puiseux monoid
		\[
			M = \bigg\langle \frac{\big\lfloor p_n/2 \big\rfloor}{p_n}, \frac{p_n - \big\lfloor p_n/2 \big\rfloor}{p_n} \ \bigg{|} \ n \in \nn \bigg\rangle.
		\]
		It is not hard to verify that $M$ is atomic and every generator given above is an atom.  Additionally, as $a \ge 1/3$ for every $a \in \mathcal{A}(M)$, it follows that $M$ is a BF-monoid, but
		$1 \in M$ has infinitely many factorizations of length 2, so
		$|\mathsf{Z}(1)| = \infty$. In particular, $M$ is a BF-monoid that fails to be an FF-monoid.
	\end{example}
	
	\begin{example}\label{e:unstablenotbf}
		Now take $\{p_n\}$ to be a strictly increasing sequence of primes such that $p_n > (n+1)^2$ for every natural $n$, and consider the Puiseux monoid
		\[
			M = \bigg\langle \frac n{p_n}, \frac{n+1}{p_n} \ \bigg{|} \ n \in \nn \bigg\rangle.
		\]
		One can check using elementary divisibility that the displayed generators are precisely the atoms of $M$ and, as a result, $M$ is atomic.  As $p_n > (n+1)^2$, there exist $a,b \in \nn_0$ such that $an + b(n+1) = p_n$. Therefore
		\[
			a+b > \frac{an + b(n+1)}{n+1} > n+1
		\]
		implies $|\mathsf{L}(1)| = \infty$. Hence, $M$ is an unstable Puiseux monoid that is not a BF-monoid.
	\end{example}
	
	\section{The set of elasticity of primary Puiseux monoids}
	\label{sec:ppmelasticity}
	
	Results in \cite[Section~2]{CHM06} state that the set of elasticities of a numerical monoid $S$ has only one limit point, namely $\rho(S)$. As the elasticity of a numerical monoid is accepted, $R(S)$ is closed, i.e., $\overline{R(S)} = R(S)$. In this section, we describe the topology of $R(M)$ for all primary Puiseux monoids $M$. In contrast to numerical monoids, many primary Puiseux monoids $M$ satisfy that $\overline{R(M)} = [1, \rho(M)]$. We start the section gathering a few technical results, and then we use them to describe $\overline{R(M)}$ in Theorems~\ref{t:finiteunstable} and~\ref{t:infiniteunstable}.
	
	\begin{example}\label{e:primarydense}
		Let $p_1, p_2, \dots$ be an increasing enumeration of the prime numbers, and let $M_1 = \langle n/p_n \mid n \in \nn \rangle$. Figure~\ref{f:primarydense} shows a plot depicting the elasticities of selected elements of $M_1$. The large red dots denote elasticities of integer elements $n \in M_1$, and the small blue dots trailing each red dot are for selected elements of the form $n + \epsilon$ for $\epsilon \in M_1$ uniquely factorable. Proposition~\ref{p:extraatom} ensures every elasticity achieved in $M_1$ occurs in this way.
	\end{example}
	
	\begin{example}\label{e:primarystable}
		Let $p_1, p_2, \dots$ be an increasing enumeration of the prime numbers, and consider the Puiseux monoid
		\[
			M_2 = \bigg\langle \frac n{p_n} \ \bigg{|} \ n \le 12 \bigg\rangle + \bigg\langle \frac{30}{p_n} \ \bigg{|} \ n > 12 \bigg\rangle = \bigg\langle \frac 12, \frac 23, \dots, \frac{12}{37} \bigg\rangle + \bigg\langle \frac{30}{41}, \frac{30}{43}, \dots \bigg\rangle.
		\]
		A plot depicting the elasticities of selected elements of $M_2$ is given in Figure~\ref{f:primarydense}. Since $|\mathsf L(30)| = \infty$, every subsequent integer has infinite elasticity. As such, every finite elasticity achieved in $M_2$ occurs as in the plot.
	\end{example}
	
	\begin{figure}
	\begin{center}
		\includegraphics[width=6.0in]{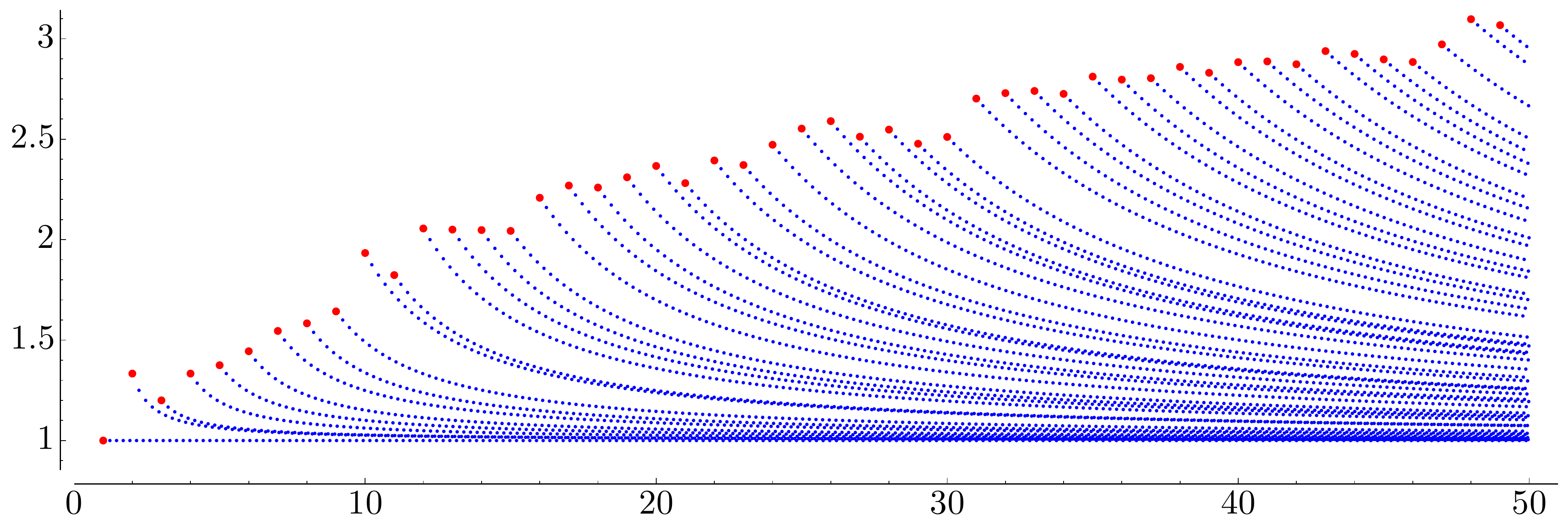} \\
		\includegraphics[width=6.0in]{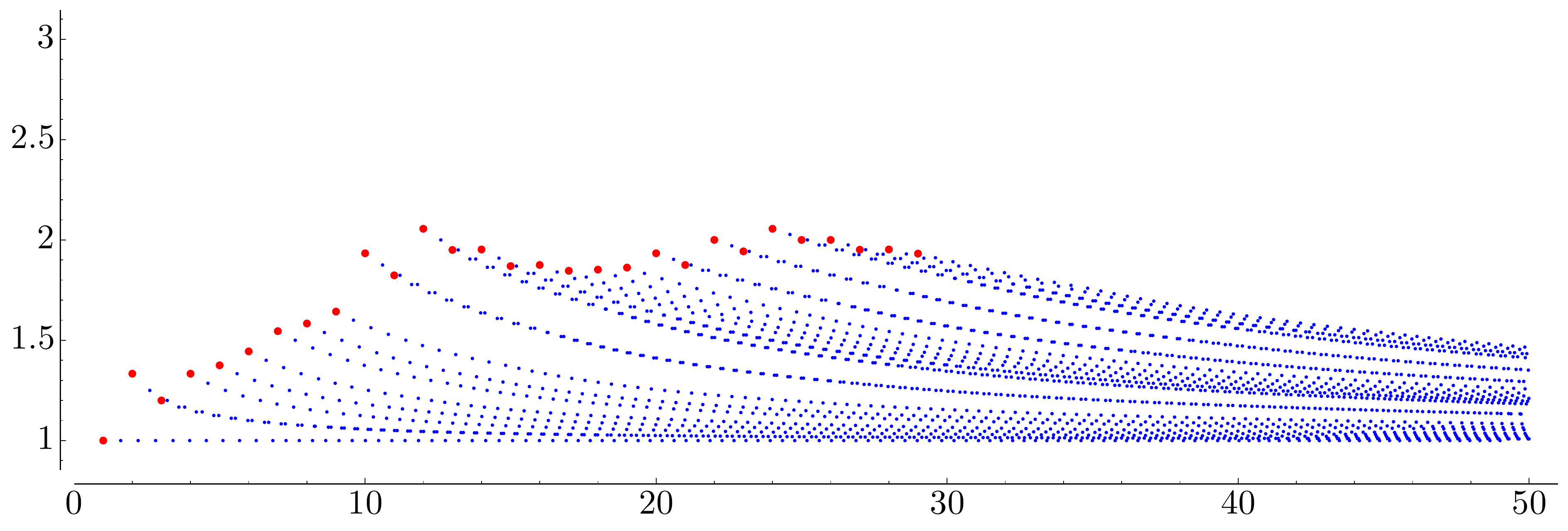}
	\end{center}
	\caption{A plot depicting the elasticities of elements of $M_1$ and $M_2$ from Examples~\ref{e:primarydense} and~\ref{e:primarystable}, respectively. In each, the large red dots denote elasticities of integer elements.}
	\label{f:primarydense}
	\end{figure}

	Proposition~\ref{p:extraatom} illustrates the main property of primary Puiseux monoids that we use to characterize their sets of elasticities (cf. \cite[Lemma~2.1]{BCCKW06}).
	\vspace{-2pt}
	\begin{proposition} \label{p:extraatom}
		Let $M$ be a primary Puiseux monoid with infinitely many atoms. For each $x \in M$, there exists $x' \in M$ with $\mathsf L(x') = \mathsf L(x) + 1$.
	\end{proposition}
	
	\begin{proof}
		Fix a factorization $a_1 + \dots + a_r \in \mathsf{Z}(x)$. Since $M$ has infinitely many atoms, some atom $a = q/p \in \mathcal A(M)$ does not appear among the $a_i$. Let $x' = x + a$, and fix a factorization $a'_1 + \dots + a'_s \in \mathsf{Z}(x')$. Applying the $p$-adic valuation map to $a_1 + \dots + a_r + a = a'_1 + \dots + a'_s$,
		we obtain that $\mathsf d(a'_i) = p$ for some index $i$.  In particular, every factorization of $x'$ must have $a$ as a formal summand, so $\mathsf L(x') = \mathsf L(x) + 1$.  
	\end{proof}
	
	\begin{proposition} \label{p:stabledecomp}
		Fix a primary Puiseux monoid $M$ and an element $x \in M$.
		\begin{enumerate}
		\item There is at most one expression of the form $x = s + u$ where $s$ is stable and uniquely factorable, and $u$ is unstable.
		\vspace{2pt}
		\item If $\mathsf{L}(x)$ is finite, then there exist unique $s \in \mathcal{S}(M)$ and $u \in \mathcal{U}(M)$ such that $x = s + u$. Moreover, $s$ is uniquely factorable and $u$ is absolutely unstable.
		\end{enumerate}
	\end{proposition}
	
	\begin{proof}
		First, suppose by contradiction that $x = s + u = s' + u'$, where $u,u' \in \mathcal{U}(M)$ and $s,s' \in \mathcal{S}(M)$ are distinct elements with unique factorizations
		\[
			\sum_{i=1}^n \alpha_i \frac{a_i}{p_i} \in \mathsf{Z}(s) \  \text{ and } \  \sum_{i=1}^n \alpha'_i \frac{a_i}{p_i} \in \mathsf{Z}(s').
		\]
		We can assume $s$ and $s'$ have no common divisors in $M$. Since $u$ and $u'$ can be factored using only unstable atoms, by applying each $p_j$-adic valuation to both expressions for $x$, we see that $s \in \nn$, at which point Lemmas~\ref{l:ppm} and~\ref{l:stable} imply $|\mathsf L(s)| = \infty$, a contradiction.
		
		Next, fix $s \in \mathcal{S}(M)$ and $u \in \mathcal{U}(M)$ such that $x = s + u$ has $\mathsf{L}(x)$ finite. Because $\mathsf{L}(s)$ is also finite, Lemma~\ref{l:stable} forces $s$ to be uniquely factorable. Then by the above argument, the given decomposition of $x$ must be unique, which in turn immediately implies that $u$ has to be absolutely unstable.
	\end{proof}
	
	Let us begin the study of $\overline{R(M)}$ for those primary Puiseux monoids $M$ having finitely many unstable atoms.
	
	\begin{theorem}\label{t:finiteunstable}
		Fix a non-finitely generated primary Puiseux monoid $M$ with finitely many unstable atoms. If $M$ contains an absolutely unstable element $u$ with $|\mathsf L(u)| > 1$, then $1$ is the only limit point of $R(M)$. Otherwise, $R(M) = \{1, \infty\}$.
	\end{theorem}
	
	\begin{proof}
		First, we claim that for every $k \in \nn$ there exists a uniquely factorable stable element $s_k$ with $\mathsf L(s_k) = \{k\}$.  Since $\mathcal{S}(M)$ is not finitely generated, there is an infinite sequence $\{a_n\}$ of stable atoms with $\mathsf{d}(a_n) < \mathsf{d}(a_{n+1})$ for each natural $n$. It is not hard to check that $ma_i$ is uniquely factorable for every $m < \mathsf{d}(a_i)$. Hence choosing $n$ large enough such that $k < \mathsf d(a_n)$, one finds that $s_k = k a_n$ satisfies $\mathsf{L}(s_k) = \{k\}$, which proves the claim.
		
		Now, given an absolutely unstable element $u \in M$ with $|\mathsf L(u)| > 1$, consider the elements $x = s_k + u$. Since $\mathcal{U}(M)$ is finitely generated, it follows that $|\mathsf L(u)| < \infty$. Take $\ell = \min \mathsf{L}(u)$ and $L = \max \mathsf{L}(u)$. Since $s_k$ is uniquely factorable, Proposition~\ref{p:stabledecomp} guarantees that the given decomposition of $x$ in $\mathcal{S}(M) + \mathcal{U}(M)$ is unique, and as such, $\min \mathsf{L}(x) = \ell + k$ and $\max \mathsf{L}(x) = L + k$.  In particular,
		\[
			\bigg\{\frac{L + k}{\ell + k} \ \bigg{|} \ k \in \nn \bigg\} \subseteq R(M).
		\]
		On the other hand, since $\mathcal{U}(M)$ is finitely generated, all but finitely many unstable elements are divisible by stable elements; let them be $U = \{u_1, \dots, u_n\}$.  Lemma~\ref{l:stable} and Proposition~\ref{p:stabledecomp} ensure that the only elements having finite elasticity are of the form $u_i + s$ for some factorial stable element $s$.  This, along with the fact that $\mathsf{L}(\mathsf{n}(s))$ is infinite for each $s \in \mathcal{S}(M)$, yields
		\[
			R(M) = \{1,\infty\} \cup \bigg\{\frac{\max \mathsf{L}(u_i) + k}{\min \mathsf{L}(u_i) + k} \ \bigg{|} \ 1 \le i \le n \ \text{ and } \ k \in \nn \bigg\}.
		\]
		At this point, both claims immediately follow.  
	\end{proof}
	
	Theorem~\ref{t:infiniteunstable} implies that, aside from numerical monoids and the primary Puiseux monoids covered by Theorem~\ref{t:finiteunstable}, the set $R(M)$ of elasticities of a primary Puiseux monoid $M$ is dense in $[1, \rho(M)]$. Let us collect one last technical result central to the proof of Theorem~\ref{t:infiniteunstable}.
	
	\begin{lemma}\label{l:density}
		Let $\{a_n\}$ and $\{b_n\}$ be two sequences in $\nn$ such that $\{a_n/b_n\}$ converges to $\ell \in \rr_{>1} \cup \{\infty\}$, and $a_n/b_n < \ell$ for each natural $n$. Then the set
		\[
			S = \bigg\{\frac{a_n + k}{b_n + k} \ \bigg{|} \ n,k \in \nn \bigg\}
		\]
		is dense in $[1,\ell]$.
	\end{lemma}
	
	\begin{proof}
		Since $\ell > 1$, we can assume without loss of generality that $a_n > b_n$ for $n \in \nn$. Replacing $\{a_n/b_n\}$ by one of its subsequences if necessary, we can assume that either $\ell = \infty$ or $\lim b_n = \infty$. Then setting $c_n = a_n - b_n$ for each natural $n$, one finds that $\lim c_n = \infty$. Suppose first that $\ell = \infty$. In this case, $\lim c_n/b_n = \infty$. Since
		\[
			\frac{a_n + k}{b_n + k} = 1 + \frac{c_n}{b_n + k},
		\]
		it suffices to show that the set 
		\[
			T = \bigg\{ \frac {b_n}{c_n} + \frac k{c_n} \ \bigg{|} \ n, k \in \nn \bigg\}
		\]
		is dense in $\rr_{\ge 0}$.  Indeed, we can find an element of $T$ within $\epsilon > 0$ of any $q \in \rr_{> 0}$ by first choosing $n$ large enough that $b_n/c_n < q$ and $1/c_n < \epsilon$, and then choosing $k$ so that $b_n/c_n + k/c_n$ is within $\epsilon$ of $q$.  
		
		Now suppose that $\ell$ is finite.  Since $\lim c_n/b_n = \ell - 1$, it suffices to show the set $T$ defined above is dense in the ray $[\frac{1}{\ell - 1}, \infty)$.  This amounts to showing that
		\[
			X = \big\{x_n + k/c_n \ \big{|} \ n \in \nn \big\}
		\]
		is dense in $\rr_{\ge 0}$, where $x_n = b_n/c_n - (\ell - 1)^{-1}$. Since $\lim x_n = 0$, we proceed as in the previous paragraph to check that $X$ is dense in $\rr_{\ge 0}$. This completes the proof.  
	\end{proof}
	
	\begin{theorem}\label{t:infiniteunstable}
		For a primary Puiseux monoid $M$ such that $\mathcal{U}(M)$ is not finitely generated, the following statements hold.
		\begin{enumerate}
			\item If $M = \mathcal U(M)$, then $R(M)$ is dense in $[1, \rho(M)]$.
			\item If $M \supsetneq \mathcal U(M)$, then $R(M)$ is dense in $[1,\rho(\mathcal{U}(M))] \cup \{\infty\}$.
		\end{enumerate}
	\end{theorem}
	
	\begin{proof}
		We begin by showing that if $\mathcal S(M) \ne \{0\}$, then 
		\begin{equation} \label{eq:set of elasticities equality}
			R(M) = R(\mathcal{U}(M)) \cup \{\infty\}.
		\end{equation}
		By assumption, Theorem~\ref{t:ppmbf} implies $\infty \in R(M)$.  Also, any $u \in \mathcal U(M)$ with $\mathsf L_M(u)$ finite is absolutely unstable by Proposition~\ref{p:stabledecomp}, so $\rho_{\mathcal{U}(M)}(u) = \rho_M(u)  \in R(M)$.  As a consequence, we have that $R(\mathcal{U}(M)) \cup \{\infty\} \subseteq R(M)$. Conversely, fix $x \in M$ with $\mathsf L_M(x)$ finite. Proposition~\ref{p:stabledecomp} states there is a unique expression $x = s + u$ with $s \in \mathcal{S}(M)$ and $u \in \mathcal{U}(M)$, and also guarantees $s$ is uniquely factorable and $u$ is absolutely unstable. Writing $\mathsf L_M(s) = \{\ell\}$, one obtains that $\mathsf L_M(x) = \mathsf L(u) + \ell$. Since $\mathcal{U}(M)$ is not finitely generated, Proposition~\ref{p:extraatom} implies some unstable $u' \in \mathcal U(M)$ satisfies $\mathsf L_{\mathcal{U}(M)}(u') = \mathsf L_M(x)$. In particular, $\rho_M(x) = \rho_{\mathcal U(M)}(u') \in R(\mathcal{U}(M))$. Thus, \eqref{eq:set of elasticities equality} holds.

		It suffices to assume $M = \mathcal U(M)$.  We claim there exists a sequence $\{r_n\} \subseteq M$ such that the sequence $\{\rho(r_n)\}$ converges to $\rho(M)$ and $\rho(r_n) < \rho(M)$ for every natural $n$. Indeed, Theorem~\ref{t:ppmbf} ensures this happens if $\rho(M) = \infty$. So assume that $\rho(M) < \infty$. The existence of such a sequence $\{r_n\}$ is also guaranteed if the elasticity of $M$ is not accepted. Finally, suppose the elasticity of $M$ is accepted and set $a = \min \mathcal{A}(M)$ and $A = \max \mathcal{A}(M)$. Take $r_n = n \mathsf{n}(a) \mathsf{n}(A) + a'$, where $a' \in \mathcal{A}(M) \! \setminus \! \{a, A\}$. Because for any $\epsilon > 0$,
		\[
			\rho(r_n) = \frac{\max \mathsf{L}(r_n)}{\min \mathsf{L}(r_n)} \ge \frac{n \mathsf{d}(a) \mathsf{n}(A) + 1}{n \mathsf{d}(A) \mathsf{n}(a) + 1} > \frac{A}{a} - \epsilon
		\]
		for all large enough $n$, the sequence $\{\rho(r_n)\}$ converges to $\rho(M)$. Also, by Corollary~\ref{c:pmaccept}, $\rho(r_n) < \rho(M)$ for every $n$.
		Now, Proposition~\ref{p:extraatom} ensures that
		\[
			\bigg\{ \frac{\mathsf{n}(\rho(r_n)) + k}{\mathsf{d}(\rho(r_n)) + k} \ \bigg{|} \ n,k \in \nn \bigg\} = \bigg\{ \frac{\max \mathsf{L}(r_n) + c_nk}{\min \mathsf{L}(r_n) + c_nk} \ \bigg{|} \ n,k \in \nn \bigg\} \subset R(M),
		\]
		where $c_n = \max \mathsf{L}(r_n) \mathsf{n}(\rho(r_n))^{-1}$. Thus, the proof follows by Lemma~\ref{l:density}.
	\end{proof}
	
	\begin{example}\label{e:infiniteunstable}
		A primary Puiseux monoid $M$ satisfying the conditions of Theorem~\ref{t:infiniteunstable} can be fully elastic, but in general need not be.  For example, let $p_1, p_2, \dots$ be an increasing enumeration of all prime numbers except $3$, and set $M = \langle n/p_n \mid n \ge 1 \rangle$. Notice that $R(M)$ contains no odd integers larger than $1$. Indeed, each integer $n \ge 2$ has max factorization length $\mathsf M(n) = p_n$ and min factorization length $\mathsf m(n) = 2n$, so~by the proof of Theorem~\ref{t:infiniteunstable},
		\[
			R(M) = \left\{\frac{p_n + k}{2n + k} \ \bigg{|} \ k \ge 0\right\}.
		\]
		In particular, $\mathsf n(r) - \mathsf d(r)$ is odd for any $r \in R(M)$ aside from $r = 1$.  
	\end{example}
	
	
	
		
	
	\section{A bifurcus Puiseux monoid}
	\label{sec:bifurcus}
	
	In this section we introduce a Puiseux monoid that is bifurcus (Definition~\ref{d:bifurcus}), thus demonstrating the broad range of behaviors exhibited by the sets of elasticities of Puiseux monoids.  
	
	\begin{definition}\label{d:bifurcus}
		An atomic monoid $M$ is \emph{bifurcus} if $2 \in \mathsf{L}(x)$ for every $x \in M \! \setminus \!\mathcal A(M)$.
	\end{definition}
	
	Suppose $\{p_{j,n} \mid j, n \ge 1\}$ is a collection of primes satisfying $p_{j,n} \ge 13$ and $p_{j,n} \ge 2^j$ for all $j, n \ge 1$. We construct a bifurcus monoid $M$ iteratively, starting with $M_0 = \langle \frac{1}{2}, \frac{1}{3} \rangle$. Enumerate the reducibles $a_{1,1}, a_{1,2}, \ldots \in M_0$ with no length~2 factorization, and let
	\[
		M_1 = M_0 + \left\langle \frac{a_{1,n}}{2} - \frac{1}{p_{1,n}}, \frac{a_{1,n}}{2} + \frac{1}{p_{1,n}} \ \bigg{|} \ n \ge 1 \right\rangle.
	\]
	Construct $M_j$ from $M_{j-1}$ in the same manner as above by enumerating the reducible elements $a_{j,1}, a_{j,2}, \ldots \in M_{j-1}$ with no length~2 factorization. Finally, let
	\[
		M = \textstyle\bigcup_{j \ge 0} M_j.
	\]
	
	\begin{theorem}\label{t:bifurcus}
		The monoid $M$ constructed above has the following properties.
		\begin{enumerate}
		\item 
		$\min M^\bullet = 1/3$.
		
		\item 
		$\mathcal A(M) = \bigcup_{j \ge 0} \mathcal A(M_j)$.
		
		\item 
		Every element of $M$ has a length~2 factorization.
		
		\end{enumerate}
		In particular, $M$ is atomic, bifurcus, and a BF-monoid.
	\end{theorem}
	
	\begin{proof}
		For each generator $u = \frac{1}{2}a_{j,n} \pm p_{j,n}^{-1}$ defined above, let $\mathsf p(u) = p_{j,n}$. 
		The smallest element of $M_0$ with no length~2 factorization is $\frac{2}{3} + \frac{1}{2}$, which means every generator in $\mathcal A(M_1) \! \setminus \! M_0$ is strictly greater than $\frac{1}{2}$ since each $p_{1,n} \ge 13$. Proceeding inductively, if the smallest generator introduced in $M_{j-1}$ is at least $\frac{1}{2}$, the smallest element of $M_j$ with no length~2 factorization must be at least $\frac{2}{3} + \frac{1}{2}$, so~every generator of $M_j \! \setminus \! M_{j-1}$ is also at least~$\frac{1}{2}$. As such, $\min M^\bullet = \frac{1}{3}$.
	
		Suppose by contradiction that $u \in \mathcal A(M_j) \! \setminus \! M_{j-1}$ is reducible in $M_{j'}$ for some $j' > j$.  Pick a factorization $z = \alpha_1 u_1 + \cdots + \alpha_m u_m \in \mathsf{Z}(u)$ in $M_{j'}$ with maximal length.  Notice that no two distinct atoms $u_i$ and $u_{i'}$ in $z$ can satisfy $\mathsf p(u_i) = \mathsf p(u_{i'})$ since $u_i + u_{i'}$ would have a length~3 factorization. Hence each atom $u_i$ in $z$ is uniquely determined by $\mathsf p(u_i)$.
	
		Comparing $\mathsf p(u)$-valuations above, we see that $\mathsf p(u)$ must appear in the denominator of some atom (which we may assume is $u_1$) in $z$.
		The key observation is that if $\mathsf p(u_1) = p_{j+1,n}$ for some $n$, then $a_{j+1,n} \ge u + \frac{2}{3}$, so
		\[
			2u_1 = a_{j+1,n} \pm 2p_{j+1,n}^{-1} \ge u + \textstyle\frac{2}{3} \pm 2p_{j+1,n}^{-1} > u.
		\]
		Inductively, this implies that if $u_1 = \frac{1}{2}a_{k,n} \pm p_{k,n}^{-1}$ for some $k, n \ge 1$, then $2^{k-j}u_1 > u$.
		
		Now, if some other atom in $z$ has $\mathsf p(u_1)$ in its denominator, we can again relabel it as $u_1$, and $2^{k-j}u_1 > u$ will still be satisfied for $\mathsf p(u_1) = p_{k,n}$ with $k$ strictly larger than before. Since only finitely many atoms appear in the factorization for $u$, this process eventually terminates in a choice for $u_1$ such that $\mathsf p(u_1) = p_{k,n}$ does not appear in the denominator of any $u_i$ for $i \ge 2$ and $2^{k-j}u_1 > u$.
		
		At this point, comparing the $\mathsf p(u_1)$-valuation of $u$ and its factorization, we can conclude $\mathsf p(u_1)$ divides $\alpha_1$.  As such,
		\[
			u \ge \alpha_1 u_1 \ge \mathsf p(u_1) u_1 > 2^{k} u_1 \ge u,
		\]
		which is a contradiction.
		
		
		For the third claim, each element $a \in M$ first appears in some $M_j$.  If $a$ has a factorization of length 2, then we are done. Otherwise $a$ has a length~2 factorization in each subsequent monoid by part~(2).
		
	
		We now conclude that $M$ is atomic and a BF-monoid by part~(1) and Theorem~\ref{p:zerolimitbf}, and the remaining claim follows from part~(3).
	\end{proof}
	
	\section*{Acknowledgments}
	
	While working on this paper, the first author was supported by the UC Berkeley Chancellor Fellowship.

\end{document}